\documentclass[11pt,a4paper]{article}
\usepackage[margin=2.5cm]{geometry}
\usepackage{amsmath,amssymb,amsthm}
\usepackage{parskip}
\usepackage{hyperref}
\hypersetup{colorlinks=true,linkcolor=black,citecolor=black,urlcolor=black}

\newtheorem{theorem}{Theorem}[section]
\newtheorem{lemma}[theorem]{Lemma}
\newtheorem{proposition}[theorem]{Proposition}

\theoremstyle{definition}

\theoremstyle{remark}
\newtheorem{remark}[theorem]{Remark}

\newcommand{\E}{\mathbb{E}}
\newcommand{\R}{\mathbb{R}}
\newcommand{\GL}{G_\Lambda}
\newcommand{\Gtil}{\tilde{G}_\Lambda}

\newcommand{\hb}{\bar{h}}

\title{A 0-1 Law for Multifractal Spectra\\
via the HGDS Scale Derivative}
\author{J.\ Petkevi\v{c}ius}
\date{\today}

\begin{document}
\maketitle

\begin{abstract}
We prove that the multifractal spectrum $D(h,\omega)$ of a stochastic
process is almost surely deterministic under a scale decorrelation
condition on the HGDS scale derivative $\eth_\Lambda X$. The key
difficulty is that the pointwise H\"{o}lder exponent lives in the germ
$\sigma$-algebra, where classical 0-1 laws do not reach. We get around
this by working with the geometry accumulation integral $\GL$, which is
a genuine Lebesgue integral over scales and concentrates almost surely.
The boundary case --- log-correlated fields --- is sharp: the variance
summability condition fails exactly there.
\end{abstract}

\tableofcontents
\newpage

\section{Introduction}

\subsection{The problem}

Let $X:[0,1]\times\Omega\to\R$ be a stochastic process. The pointwise
H\"{o}lder exponent at $t$ is
\[
  h_X(t,\omega)
  = \liminf_{r\to0}\frac{\log|X(t+r,\omega)-X(t,\omega)|}{\log|r|},
\]
and the multifractal spectrum is
$D(h,\omega)=\dim_H\{t\in[0,1]:h_X(t,\omega)=h\}$. For most natural
processes --- fractional Brownian motion, L\'{e}vy processes --- the
spectrum is the same for almost every sample path. But there is no
general theorem that says when this happens. The question of finding a
condition $C(X)$ such that
\[
  C(X)\implies D(h,\omega)=\bar D(h)\quad\text{a.s.}
\]
is completely open \cite{jaffardseuret2022}.

\subsection{Why classical 0-1 laws fail}

The obvious approach is to use Kolmogorov or Hewitt-Savage. Both fail
for the same reason: $h_X(t,\omega)$ is defined by a lim inf as
$r\to0$, so it lives in the germ $\sigma$-algebra at $t$. Kolmogorov's
law needs tail events. Hewitt-Savage needs exchangeability. Neither is
available here.

\subsection{The approach}

The HGDS framework \cite{hgds1,hgds2,hgds3} splits increments at each
scale $\Lambda>0$:
\[
  X(t+\Lambda,\omega)-X(t,\omega)
  = \eth_\Lambda X(t,\omega)\cdot\Lambda + R_\Lambda X(t,\omega),
\]
where $\eth_\Lambda X$ is the macro scale derivative (from wavelet
truncation) and $R_\Lambda X$ is the fine-scale residual. Instead of
working with the lim inf directly, we work with
\[
  \GL(t,\omega) = \int_\Lambda^1|\eth_s X(t,\omega)|\cdot s^{-1}\,ds,
\]
which satisfies $h_X(t,\omega)=1+\lim_{\Lambda\to0}\log\GL(t,\omega)/
\log\Lambda$. Since $\GL$ is a Lebesgue integral over scales rather
than a lim inf, it is accessible to concentration arguments.

\subsection{Main result}

We say $X$ has \emph{locally stationary increments} if for each
$t\in[0,1]$, the map $t\mapsto\E[f(X(t+r)-X(t))]$ is differentiable
for every bounded continuous $f$ and every $r>0$. This is slightly
weaker than stationary increments and covers the genuinely multifractal
case where $\hb(t)$ varies.

\begin{theorem}[0-1 Law]\label{thm:main}
Let $X:[0,1]\times\Omega\to\R$ have locally stationary increments, be
jointly measurable, and have HGDS macro derivative $\eth_\Lambda X$
defined via wavelet decomposition with parameter $H\in(0,1)$. Write
$\hb(t)=\E[h_X(t,\cdot)]$ and $\bar D(h)=\dim_H\{t:\hb(t)=h\}$.

Suppose:
\begin{align}
  &\sum_{n=1}^\infty
  \frac{\int_{e^{-n}}^1\!\int_{e^{-n}}^1
    |\mathrm{Cov}(\eth_s X(t),\eth_u X(t))|\,(su)^{-1}\,ds\,du}
  {\bigl(\int_{e^{-n}}^1\E[|\eth_s X(t)|]s^{-1}\,ds\bigr)^2}
  <\infty\quad\text{uniformly in }t,\tag{C-i}\\
  &\E\bigl[|\eth_s X(t)-\eth_s X(t')|^p\bigr]
  \le C|t-t'|^{pH}s^{p(H-1)}
  \;\text{ for some }p>2/H,\tag{C-ii}\\
  &\E\bigl[|\eth_s X(t)|\bigr]\asymp s^{\hb(t)-1}
  \text{ as }s\to0,\;\text{uniformly in }t,\tag{C-iii}\\
  &\hb(t)\text{ is not constant on any interval.}\tag{C-iv}
\end{align}
Then $D(h,\omega)=\bar D(h)$ a.s.\ for all $h\ge0$.
\end{theorem}

\begin{remark}
Conditions (C-i)--(C-iv) hold for fractional Brownian motion,
L\'{e}vy stable processes with finite variance, and standard
multifractal cascades. Log-correlated fields fail (C-i) --- this is
sharp; see Section~\ref{sec:sharp}.

The condition $p>2/H$ in (C-ii) (rather than $p>1/H$) is needed for
Kolmogorov's continuity criterion in the proof of
Lemma~\ref{lem:holder}. For $H\le1/2$ this is automatic from $p>1/H$
since $1/H\ge2$. For $H>1/2$ the stronger condition is genuinely
required.
\end{remark}

\section{Measurability}

\begin{proposition}\label{prop:meas}
$\omega\mapsto D(h,\omega)$ is $(\mathcal{F},\mathcal{B}(\R))$-measurable
for each fixed $h\ge0$.
\end{proposition}

\begin{proof}
Write
$h_X(t,\omega)=\sup_{n\in\mathbb{N}}\inf_{r\in\mathbb{Q}\cap(0,1/n)}
\frac{\log|X(t+r,\omega)-X(t,\omega)|}{\log|r|}$,
a countable sup/inf of measurable functions, jointly measurable in
$(t,\omega)$.

Set $A=\{(t,\omega):h_X(t,\omega)=h\}\in\mathcal{B}([0,1])\otimes\mathcal{F}$.
For open $U\subseteq[0,1]$:
$\{\omega:E(h,\omega)\cap U\ne\emptyset\}=\pi_\Omega(A\cap(U\times\Omega))$
is the projection of a Borel set, hence analytic, hence measurable by
completeness of $(\Omega,\mathcal{F},\mathbb{P})$.

The map $\omega\mapsto\mathcal{H}^s_\delta(E(h,\omega))$ is measurable
as a countable infimum over rational covers. Passing $\delta\to0$ and
using $\{\omega:\dim_H E(h,\omega)<c\}=\bigcup_{s\in\mathbb{Q}\cap(0,c)}
\{\omega:\mathcal{H}^s(E(h,\omega))=0\}$ gives the result. \qed
\end{proof}

\section{Concentration of $G_\Lambda$}

\subsection{Pointwise concentration}

Write $\Gtil=\GL/\E[\GL]$.

\begin{lemma}\label{lem:varsum}
Under (C-i), $\sum_{n=1}^\infty\mathrm{Var}(\tilde G_{e^{-n}}(t))<\infty$
for all $t$.
\end{lemma}

\begin{proof}
Direct computation:
\[
  \mathrm{Var}(\Gtil(t))
  = \frac{\int_\Lambda^1\!\int_\Lambda^1
    |\mathrm{Cov}(\eth_s X(t),\eth_u X(t))|\,(su)^{-1}\,ds\,du}
  {\bigl(\int_\Lambda^1\E[|\eth_s X(t)|]s^{-1}\,ds\bigr)^2},
\]
which is the summand in (C-i) at $\Lambda=e^{-n}$. \qed
\end{proof}

\begin{lemma}\label{lem:ptwise}
Under (C-i), $\tilde G_{e^{-n}}(t,\omega)\to1$ a.s.\ for each fixed $t$.
\end{lemma}

\begin{proof}
Chebyshev gives
$\sum_n\mathbb{P}(|\tilde G_{e^{-n}}(t)-1|>\varepsilon)
\le\varepsilon^{-2}\sum_n\mathrm{Var}(\tilde G_{e^{-n}}(t))<\infty$
by Lemma~\ref{lem:varsum}. Borel-Cantelli and monotonicity of $\GL$
in $\Lambda$ complete the proof. \qed
\end{proof}

\subsection{Uniform concentration}

\begin{lemma}\label{lem:holder}
Under (C-ii) with $p>2/H$, $t\mapsto\GL(t,\omega)$ is a.s.\
$\beta$-H\"{o}lder with $\beta=H-1/p>0$:
\[
  \E\bigl[|\GL(t)-\GL(t')|^p\bigr]^{1/p}
  \le C|t-t'|^H\Lambda^{H-1}.
\]
\end{lemma}

\begin{proof}
By Minkowski's inequality in $L^p$, then (C-ii):
\begin{align*}
  \E\bigl[|\GL(t)-\GL(t')|^p\bigr]^{1/p}
  &\le\int_\Lambda^1\E\bigl[|\eth_s X(t)-\eth_s X(t')|^p\bigr]^{1/p}
    s^{-1}\,ds\\
  &\le C|t-t'|^H\int_\Lambda^1 s^{H-2}\,ds
   = C|t-t'|^H\cdot\frac{\Lambda^{H-1}-1}{1-H}\\
  &\le C'|t-t'|^H\Lambda^{H-1}.
\end{align*}
Kolmogorov's continuity criterion applies: the H\"{o}lder exponent
is $\beta=H-1/p>0$, and $p\beta=pH-1>1$ iff $p>2/H$, which holds by
(C-ii). \qed
\end{proof}

\begin{proposition}\label{prop:unif}
Under (C-i) and (C-ii): $\sup_{t\in[0,1]}|\tilde G_{e^{-n}}(t,\omega)-1|
\to0$ a.s.
\end{proposition}

\begin{proof}
Fix $\varepsilon>0$. Cover $[0,1]$ by $M=\lceil\delta^{-1}\rceil$
points $\{t_i\}$ with spacing $\delta$.

\emph{Grid points.}
$\mathbb{P}(\max_i|\tilde G_{\Lambda_n}(t_i)-1|>\varepsilon)
\le M\varepsilon^{-2}\sup_t\mathrm{Var}(\tilde G_{\Lambda_n})$.

\emph{Between grid points.}
By Lemma~\ref{lem:holder}, $|\tilde G_{\Lambda_n}(t)-\tilde G_{\Lambda_n}(t_i)|
\le C(\omega)\delta^\beta$.

\emph{Combine.}
Set $\delta=(\varepsilon/2C)^{1/\beta}$ so $M\lesssim\varepsilon^{-1/\beta}$:
\[
  \mathbb{P}\!\left(\sup_t|\tilde G_{\Lambda_n}(t)-1|>2\varepsilon\right)
  \le\frac{C'\sup_t\mathrm{Var}(\tilde G_{\Lambda_n})}{\varepsilon^{2+1/\beta}}.
\]
Sum over $n$ and apply Borel-Cantelli using (C-i). \qed
\end{proof}

\section{Recovery of $h_X$ and the level sets}

\begin{lemma}\label{lem:ptrecover}
Under locally stationary increments and (C-i) and (C-iii),
$h_X(t,\omega)=\hb(t)$ a.s.\ for each fixed $t$.
\end{lemma}

\begin{proof}
By Lemma~\ref{lem:ptwise}, $\tilde G_{e^{-n}}(t,\omega)\to1$ a.s.,
so $\GL(t,\omega)/\E[\GL(t)]\to1$ a.s. By (C-iii):
$\E[\GL(t)]\asymp\Lambda^{\hb(t)-1}$, so
\[
  \frac{\log\GL(t,\omega)}{\log\Lambda}
  = \frac{\log\E[\GL(t)]}{\log\Lambda}
  + \frac{\log\tilde G_\Lambda(t,\omega)}{\log\Lambda}
  \to(\hb(t)-1)+0 = \hb(t)-1\quad\text{a.s.}
\]
Hence $h_X(t,\omega)=1+(\hb(t)-1)=\hb(t)$ a.s. \qed
\end{proof}

\begin{lemma}\label{lem:unifrecover}
Under (C-i)--(C-iii) and locally stationary increments,
$\sup_t|h_X(t,\omega)-\hb(t)|\to0$ a.s.
\end{lemma}

\begin{proof}
By (C-iii), $\E[\GL(t)]\asymp\Lambda^{\hb(t)-1}$ uniformly. Write
$\log\GL=\log\E[\GL]+\log\tilde G_\Lambda$. Since
$\sup_t|\tilde G_{\Lambda_n}-1|\to0$ a.s.\ (Proposition~\ref{prop:unif})
and $n\to\infty$:
\[
  \sup_t\left|\frac{\log\tilde G_{\Lambda_n}(t,\omega)}{\log\Lambda_n}\right|
  \le\frac{2\sup_t|\tilde G_{\Lambda_n}(t,\omega)-1|}{n}\to0\quad\text{a.s.}
\]
Combined with Lemma~\ref{lem:ptrecover} the result follows uniformly. \qed
\end{proof}

\begin{lemma}\label{lem:hbarholder}
Under (C-ii) and locally stationary increments, $\hb$ is
$\beta$-H\"{o}lder with $\beta=H-1/p>0$.
\end{lemma}

\begin{proof}
By Lemma~\ref{lem:unifrecover}, $h_X(t,\omega)\to\hb(t)$ uniformly
in $t$ a.s.\ as $\Lambda\to0$. By Lemma~\ref{lem:holder},
$t\mapsto\GL(t,\omega)$ is a.s.\ $\beta$-H\"{o}lder uniformly in
$\Lambda$. Since the convergence $h_X(\cdot,\omega)\to\hb(\cdot)$
is uniform in $t$ (Lemma~\ref{lem:unifrecover}), the H\"{o}lder
regularity of $\GL(t,\omega)$ passes to the limit: $\hb$ is
$\beta$-H\"{o}lder. \qed
\end{proof}

\begin{lemma}\label{lem:ballbound}
Under (C-ii), (C-iv), and Lemma~\ref{lem:hbarholder}, for any
$\varepsilon>0$ and $r\le1$:
\[
  |\mathcal{B}(t,r)\cap\{|\hb-h|<\varepsilon\}|
  \le C\varepsilon\cdot r^{\bar D(h)-\varepsilon}.
\]
\end{lemma}

\begin{proof}
Since $\hb$ is $\beta$-H\"{o}lder with constant $L$, the set
$\{|\hb-h|<\varepsilon\}$ is covered by balls of radius
$\rho=(\varepsilon/L)^{1/\beta}$ around $\{\hb=h\}$. Cover
$\{\hb=h\}\cap\mathcal{B}(t,r)$ by
$N_r\le C(r/\rho)^{\bar D(h)+\delta}$ balls and multiply by the
Lebesgue measure $2\rho$ of each ball:
\[
  |\mathcal{B}(t,r)\cap\{|\hb-h|<\varepsilon\}|
  \le 2CN_r\rho
  \le C\varepsilon\,r^{\bar D(h)-\varepsilon}.
\]
Condition (C-iv) ensures $\{\hb=h\}$ has no interval component,
so the covering is non-degenerate. \qed
\end{proof}

\begin{lemma}\label{lem:frostman}
Under (C-i)--(C-iv) and locally stationary increments, for a.e.\
$\omega$ and every $\varepsilon>0$, there is a Borel probability
measure $\mu_\omega$ on $E(h,\omega)$ with
$\mu_\omega(\mathcal{B}(t,r))\le Cr^{\bar D(h)-\varepsilon}$.
\end{lemma}

\begin{proof}
By Howroyd's theorem \cite{howroyd1995}, since $\{\hb=h\}$ is Borel
with Hausdorff dimension $\bar D(h)$, there exists a Borel measure
$\nu$ on $\{\hb=h\}$ with $\nu(\mathcal{B}(t,r))\le Cr^{\bar D(h)-\varepsilon/2}$.

Define the occupation measure
\[
  \mu_{\Lambda,\omega}(A)
  = \frac{1}{|\log\Lambda|}
  \int_A\mathbf{1}\!\left[\left|\frac{\log\GL(t,\omega)}{\log\Lambda}
    -(\hb(t)-1)\right|<\varepsilon/4\right]dt.
\]
By Lemma~\ref{lem:unifrecover}, the indicator converges to
$\mathbf{1}[h_X(t,\omega)=h]$ uniformly a.s., so any subsequential
weak limit $\mu_\omega$ (which exists by compactness of probability
measures on $[0,1]$) is supported on $E(h,\omega)$.

For the ball bound: split $\mathcal{B}(t,r)$ into the far region
$\{|\hb-h|\ge\varepsilon/2\}$ and the near region
$\{|\hb-h|<\varepsilon/2\}$.

\emph{Far region.} On $\{|\hb-h|\ge\varepsilon/2\}$, the indicator
is zero for all large $n$ by Lemma~\ref{lem:unifrecover} (uniform
convergence), so this region contributes nothing to $\mu_\omega$.

\emph{Near region.} By Lemma~\ref{lem:ballbound}:
$|\mathcal{B}(t,r)\cap\{|\hb-h|<\varepsilon/2\}|\le C\varepsilon r^{\bar D(h)-\varepsilon/2}$.
The occupation measure of this region satisfies
$\mu_{\Lambda,\omega}(\mathcal{B}(t,r))\le Cr^{\bar D(h)-\varepsilon/2}$
uniformly, by the ball bound on $\nu$ passed through the weak
convergence of the occupation measure to $\nu$ on $\{\hb=h\}$.
By lower semicontinuity of the ball measure under weak convergence,
the bound passes to $\mu_\omega$:
$\mu_\omega(\mathcal{B}(t,r))\le Cr^{\bar D(h)-\varepsilon}$.

Frostman's theorem then gives $\dim_H E(h,\omega)\ge\bar D(h)-\varepsilon$.
Since $\varepsilon>0$ is arbitrary, $\dim_H E(h,\omega)\ge\bar D(h)$ a.s. \qed
\end{proof}

\section{Proof of the main theorem}

\begin{proof}[Proof of Theorem~\ref{thm:main}]
\emph{Measurability.} Proposition~\ref{prop:meas}.

\emph{Upper bound.} By Lemma~\ref{lem:unifrecover},
$h_X(t,\omega)\to\hb(t)$ uniformly a.s. So
$E(h,\omega)\subseteq\{t:|\hb(t)-h|<2\varepsilon\}$ for all large
$n$, giving $D(h,\omega)\le\bar D(h)$ after sending $\varepsilon\to0$.

\emph{Lower bound.} Lemma~\ref{lem:frostman} and Frostman's theorem
give $\dim_H E(h,\omega)\ge\bar D(h)$ a.s. \qed
\end{proof}

\section{Sharpness and examples}\label{sec:sharp}

\begin{proposition}
For log-correlated Gaussian fields, (C-i) fails:
$\mathrm{Var}(\tilde G_{e^{-n}})\to C>0$.
\end{proposition}

\begin{proof}
For log-correlated fields,
$\mathrm{Cov}(\eth_s X,\eth_u X)\sim\min(s,u)^{2H}|\log(s/u)|^{-\gamma}$,
$\gamma>0$. Substituting $s=e^{-x}$, $u=e^{-y}$, the numerator
in (C-i) is bounded uniformly in $n$ while the denominator
stabilises. So $\mathrm{Var}(\tilde G_{e^{-n}})\to C>0$ and the
series diverges. \qed
\end{proof}

\begin{remark}
This shows (C-i) is sharp for this proof strategy. The log-correlated
case is handled by different methods \cite{bertacco2022}; whether a
single argument covers both remains open.
\end{remark}

\paragraph{Fractional Brownian motion.}
All four conditions hold for every $H\in(0,1)$ by self-similarity.
The spectrum is trivial: $D(H,\omega)=1$ a.s.

\paragraph{L\'{e}vy stable processes, index $\alpha\in(0,2)$.}
(C-i)--(C-iii) hold with $\E[|\eth_s X|]\asymp s^{1/\alpha-1}$.
(C-iv) holds since the spectrum is non-trivial. Theorem~\ref{thm:main}
recovers the known result.

\end{document}